\documentclass[12pt]{article}
\usepackage{amssymb}
\usepackage{amsmath,amsthm}
\usepackage[latin1]{inputenc}
\usepackage{hyperref}
\usepackage{tikz}
\usepackage{xifthen}
\usepackage{enumerate}

\hypersetup{colorlinks=true}

\hypersetup{colorlinks=true, linkcolor=blue, citecolor=blue,urlcolor=blue}

 \newtheorem{remark}{Remark}

 \newtheorem{lemma}[remark]{Lemma}
 \newtheorem{theorem}[remark]{Theorem}
 
 \newtheorem{corollary}[remark]{Corollary}

\newcommand{\Sup}{\operatorname{Sup}}

\title{On Generalized Sierpi\'{n}ski Graphs}

\author{Juan A. Rodr\'{\i}guez-Vel\'{a}zquez$^{(1)}$, Erick D. Rodr\'{i}guez-Bazan$^{(2)}$,\\ Alejandro Estrada-Moreno$^{(1)}$
\\
$^{(1)}${\small Departament d'Enginyeria Inform\`atica i Matem\`atiques,}\\
{\small Universitat Rovira i Virgili,}  {\small Av. Pa\"{\i}sos
Catalans 26, 43007 Tarragona, Spain.} \\{\small
juanalberto.rodriguez\@@urv.cat, alejandro.estrada\@@urv.cat}
\\
$^{(2)}${\small Department of Matemathics,}\\
{\small Central University of Las Villas,}{ \small   Carretera a Camajuan\'{i} km. $5\frac{1}{2}$. Villa Clara, Cuba.} \\{\small
erickrodriguezbazan\@@gmail.com}
}

\begin{document}
\maketitle

\begin{abstract}
In this paper we obtain closed formulae for several parameters of generalized Sierpi\'{n}ski graphs $S(G,t)$ in terms of parameters of the base graph $G$.
In particular, we focus on the chromatic,  vertex cover,  clique and  domination numbers.
\end{abstract}



\section{Introduction}
Let $G=(V,E)$ be a non-empty graph of order $n$. We denote by $V^t $ the set of words of size $t$ on alphabet $V $. The letters of a word $u$ of length $t$ are denoted by $u_1u_2...u_t$. The concatenation of two words $u$ and $v$  is denoted by $uv$. Kla\v{v}ar and Milutinovi\'c introduced in \cite{Klavzar1997} the graph  $S(K_n, t)$ whose vertex set is $V^t$, where
$\{u,v\}$ is an edge if and only if there exists $i\in \{1,...,t\}$ such that:\\
$$ \mbox{ (i) }  u_j=v_j, \mbox{ if } j<i; \mbox{ (ii) } u_i\ne v_i; \mbox{ (iii) } u_j=v_i \mbox{ and } v_j=u_i  \mbox{ if } j>i.$$

When
$n = 3$, those graphs are exactly Tower of Hanoi graphs. Later, those graphs have been
called Sierpi\'{n}ski graphs in \cite{Klavzar2002} and they were studied by now from numerous points of view. The reader is
invited to read, for instance, the following recent papers \cite{Gravier...Parreau,Hinz-Parisse,Hirtz-Holz,Klavzar2002,Klavsar-Peterin,Klavsar-Zeljic} and references
therein.
This construction was generalized in \cite{GeneralizedSierpinski} for any graph $G=(V,E)$, by defining the $t$-th \emph{generalized Sierpi\'{n}ski graph} of $G$, denoted by  $S(G,t)$,  as the graph with vertex set $V^t$ and edge set defined as follows. $\{u,v\}$ is an edge if and only if there exists $i\in \{1,...,t\}$ such that:
$$ \mbox{ (i) }  u_j=v_j, \mbox{ if } j<i; \mbox{ (ii) } u_i\ne v_i \mbox{ and } \{u_i,v_i\}\in E;  \mbox{ (iii) } u_j=v_i \mbox{ and } v_j=u_i  \mbox{ if } j>i.$$



\begin{figure}[h]
\centering
\begin{tikzpicture}[transform shape, inner sep = .5mm]
\def\side{.5};
\pgfmathsetmacro\radius{\side/sqrt(3)};
\foreach \ind in {3,4,5}
{
\pgfmathparse{150-120*(\ind-3)};
\node [draw=black, shape=circle, fill=black] (\ind) at (\pgfmathresult:\radius cm) {};
}
\node [draw=black, shape=circle, fill=black] (1) at ([yshift=\side cm]3) {};
\node [draw=black, shape=circle, fill=black] (2) at ([yshift=\side cm]4) {};
\node [draw=black, shape=circle, fill=black] (6) at ([xshift=-\side cm]5) {};
\node [draw=black, shape=circle, fill=black] (7) at ([shift=({135:\side cm})]6) {};
\foreach \ind in {2,4,5}
{
\node [scale=.8] at ([xshift=.3 cm]\ind) {$\ind$};
}
\foreach \ind in {1,3,6,7}
{
\node [scale=.8] at ([xshift=-.3 cm]\ind) {$\ind$};
}
\foreach \u/\v in {1/3,2/4,3/4,3/5,4/5,5/6,6/7}
{
\draw (\u) -- (\v);
}

\def\widenodetwo{.4};
\pgfmathparse{15*\side};
\node (center) at (\pgfmathresult cm,0) {};
\pgfmathsetmacro\sideTwo{5*\side};
\pgfmathsetmacro\radiuscenter{\sideTwo/sqrt(3)};
\foreach \d in {3,4,5}
{
\pgfmathparse{150-120*(\d-3)};
\node (center\d) at ([shift=({\pgfmathresult:\radiuscenter cm})]center) {};
\foreach \ind in {3,4,5}
{
\pgfmathparse{150-120*(\ind-3)};
\node [draw=black, shape=circle, fill=black,inner sep = \widenodetwo mm] (\d\ind) at ([shift=({\pgfmathresult:\radius cm})]center\d) {};
}
\node [draw=black, shape=circle, fill=black,inner sep = \widenodetwo mm] (\d1) at ([yshift=\side cm]\d3) {};
\node [draw=black, shape=circle, fill=black,inner sep = \widenodetwo mm] (\d2) at ([yshift=\side cm]\d4) {};
\node [draw=black, shape=circle, fill=black,inner sep = \widenodetwo mm] (\d6) at ([xshift=-\side cm]\d5) {};
\node [draw=black, shape=circle, fill=black,inner sep = \widenodetwo mm] (\d7) at ([shift=({135:\side cm})]\d6) {};
\foreach \ind in {2,4,5}
{
\node [scale=.6] at ([xshift=.3 cm]\d\ind) {$\d\ind$};
}
\foreach \ind in {1,3,6,7}
{
\node [scale=.6] at ([xshift=-.3 cm]\d\ind) {$\d\ind$};
}
\foreach \u/\v in {1/3,2/4,3/4,3/5,4/5,5/6,6/7}
{
\draw (\d\u) -- (\d\v);
}
}
\foreach \d in {1,2,6,7}
{
\ifthenelse{\d=1}
{
\node (center1) at ([shift=({0,\sideTwo})]center3) {};
}
{
\ifthenelse{\d=2}
{
\node (center2) at ([shift=({0,\sideTwo})]center4) {};
}
{
\ifthenelse{\d=6}
{
\node (center6) at ([shift=({-\sideTwo,0})]center5) {};
}
{
\node (center7) at ([shift=({135:\sideTwo cm})]center6) {};
};
};
};
\foreach \ind in {3,4,5}
{
\pgfmathparse{150-120*(\ind-3)};
\node [draw=black, shape=circle, fill=black,inner sep = \widenodetwo mm] (\d\ind) at ([shift=({\pgfmathresult:\radius})]center\d) {};
}

\node [draw=black, shape=circle, fill=black,inner sep = \widenodetwo mm] (\d1) at ([yshift=\side cm]\d3) {};
\node [draw=black, shape=circle, fill=black,inner sep = \widenodetwo mm] (\d2) at ([yshift=\side cm]\d4) {};
\node [draw=black, shape=circle, fill=black,inner sep = \widenodetwo mm] (\d6) at ([xshift=-\side cm]\d5) {};
\node [draw=black, shape=circle, fill=black,inner sep = \widenodetwo mm] (\d7) at ([shift=({135:\side cm})]\d6) {};
\foreach \ind in {2,4,5}
{
\node [scale=.6] at ([xshift=.3 cm]\d\ind) {$\d\ind$};
}
\foreach \ind in {1,3,6,7}
{
\node [scale=.6] at ([xshift=-.3 cm]\d\ind) {$\d\ind$};
}
\foreach \u/\v in {1/3,2/4,3/4,3/5,4/5,5/6,6/7}
{
\draw (\d\u) -- (\d\v);
}
}
\foreach \u/\v in {1/3,2/4,3/4,3/5,4/5,5/6,6/7}
{
\ifthenelse{\u=1}
{
\draw (\u\v) -- (\v\u);
}
{
\ifthenelse{\u=3 \AND \v=5 \OR \u=6}
{
\draw (\u\v) to[bend right] (\v\u);
}
{
\draw (\u\v) to[bend left] (\v\u);
};
};
}
\end{tikzpicture}
\caption{A graph $G$ and the generalized Sierpi\'{n}ski graph $S(G,2)$ }
\label{FigSierpinski1y2}
\end{figure}
Figure \ref{FigSierpinski1y2} shows a graph $G$ and the generalized Sierpi\'{n}ski graph $S(G,2)$, while Figure \ref{FigS(G,3)} shows the Sierpi\'{n}ski graph $S(G,3)$.

Notice that if $\{u,v\}$ is an edge of $S(G,t)$, there is an edge $\{x,y\}$ of $G$ and a word $w$ such that $u=wxyy...y$ and $v=wyxx...x$. In general, $S(G,t)$ can be constructed recursively from $G$ with the following process: $S(G,1)=G$ and, for $t\ge 2$, we copy $n$ times $S(G, t-1)$ and add the letter $x$ at the beginning of each label of the vertices belonging to  the copy of $S(G,t-1)$ corresponding to $x$. Then for every edge $\{x,y\}$ of $G$, add an edge between vertex $xyy...y$ and vertex $yxx...x$. See, for instance, Figure \ref{FigS(G,3)}. Vertices of the form $xx...x$ are called \textit{extreme vertices} of $S(G,t)$. Notice that for any graph $G$ of order $n$ and any integer $t\ge 2$,  $S(G,t)$  has $n$ extreme vertices and, if $x$ has degree $d(x)$ in $G$, then the extreme vertex $xx...x$ of $S(G,t)$   also has degree  $d(x)$. Moreover,   the degrees of two vertices $yxx...x$ and  $xyy...y$, which connect two copies of $S(G,t-1)$, are  equal to  $d(x)+1$ and $d(y)+1$, respectively.

For any $w\in V^{t-1}$ and $t\ge 2$  the subgraph $\langle V_w \rangle$ of $S(G,t)$, induced by $V_w=\{wx:\; x\in V\}$, is isomorphic to $G$. Notice that there exists only one vertex $u\in V_w$ of the form $w'xx\ldots x$, where $w'\in V^{r}$ for some $r\le t-2$. We will say that $w'xx\ldots x$ is \textit{the extreme vertex} of $\langle V_w \rangle$, which is an extreme vertex in $S(G,t)$ whenever $r=0$. By definition of $S(G,t)$ we deduce the following remark.

\begin{remark}
Let $G=(V,E)$ be a graph, let $t\ge 2$ be an integer and $w\in V^{t-1}$.  If  $u\in V_w$ and $v\in V^t-V_w$ are adjacent in $S(G,t)$, then   either $u$ is the extreme vertex of  $\langle V_w \rangle$ or $u$ is adjacent to the extreme vertex  of  $\langle V_w \rangle$.
\end{remark}


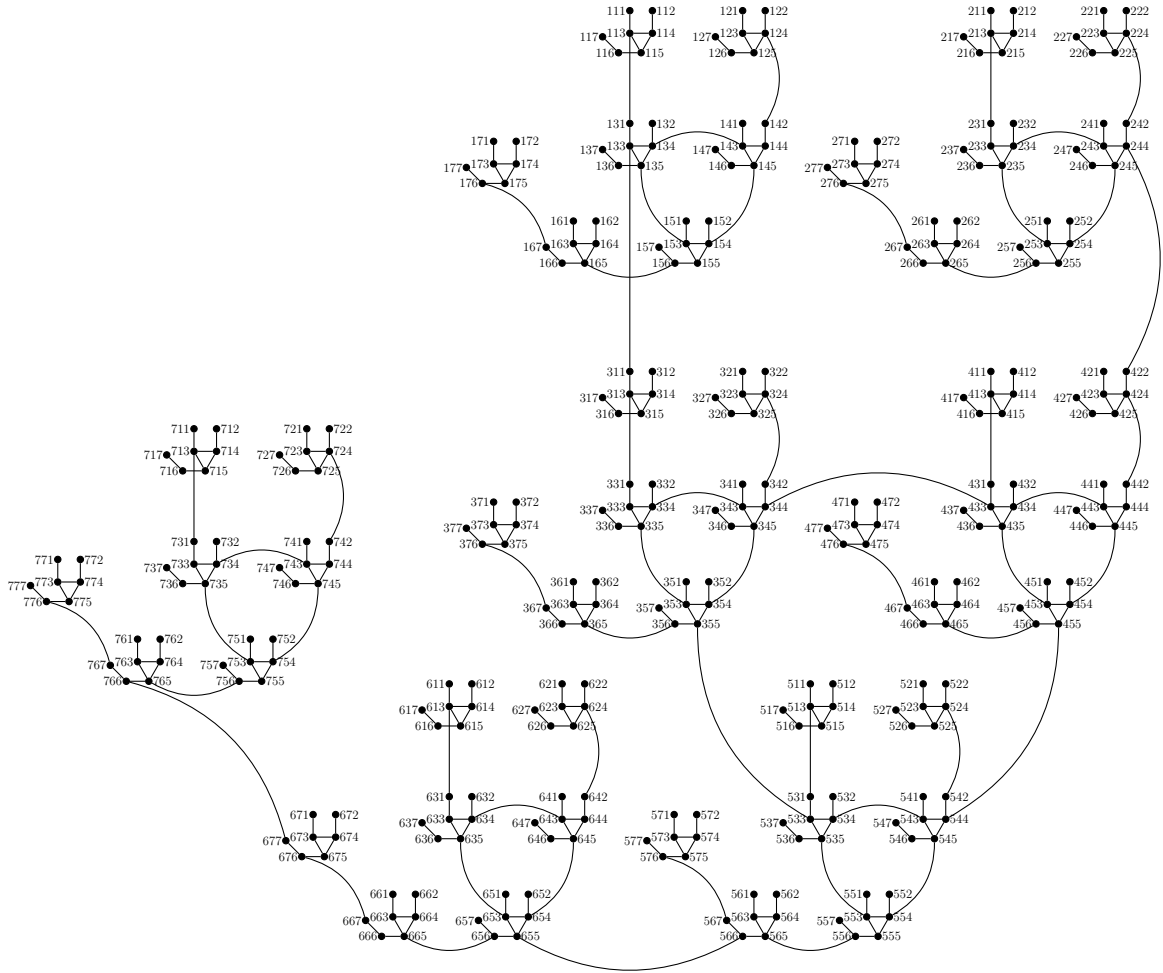
\begin{figure}[h]
\centering
\begin{tikzpicture}[transform shape, inner sep = .3mm]
\def\side{.3};
\pgfmathsetmacro\radius{\side/sqrt(3)};
\pgfmathsetmacro\sideTwo{5*\side};
\pgfmathsetmacro\radiuscenter{\sideTwo/sqrt(3)};
\pgfmathsetmacro\sideThree{16*\side};
\pgfmathsetmacro\radiuscenterThree{\sideThree/sqrt(3)};
\node (center) at (0,0) {};
\foreach \c in {3,4,5}
{
\pgfmathparse{150-120*(\c-3)};
\node (center\c) at ([shift=({\pgfmathresult:\radiuscenterThree cm})]center) {};
\foreach \d in {3,4,5}
{
\pgfmathparse{150-120*(\d-3)};
\node (center\c\d) at ([shift=({\pgfmathresult:\radiuscenter cm})]center\c) {};
\foreach \ind in {3,4,5}
{
\pgfmathparse{150-120*(\ind-3)};
\node [draw=black, shape=circle, fill=black] (\c\d\ind) at ([shift=({\pgfmathresult:\radius cm})]center\c\d) {};
}
\node [draw=black, shape=circle, fill=black] (\c\d1) at ([yshift=\side cm]\c\d3) {};
\node [draw=black, shape=circle, fill=black] (\c\d2) at ([yshift=\side cm]\c\d4) {};
\node [draw=black, shape=circle, fill=black] (\c\d6) at ([xshift=-\side cm]\c\d5) {};
\node [draw=black, shape=circle, fill=black] (\c\d7) at ([shift=({135:\side cm})]\c\d6) {};
\foreach \ind in {2,4,5}
{
\node [scale=.4] at ([xshift=.18 cm]\c\d\ind) {$\c\d\ind$};
}
\foreach \ind in {1,3,6,7}
{
\node [scale=.4] at ([xshift=-.18 cm]\c\d\ind) {$\c\d\ind$};
}
\foreach \u/\v in {1/3,2/4,3/4,3/5,4/5,5/6,6/7}
{
\draw (\c\d\u) -- (\c\d\v);
}
}
\foreach \d in {1,2,6,7}
{
\ifthenelse{\d=1}
{
\node (center\c1) at ([shift=({0,\sideTwo})]center\c3) {};
}
{
\ifthenelse{\d=2}
{
\node (center\c2) at ([shift=({0,\sideTwo})]center\c4) {};
}
{
\ifthenelse{\d=6}
{
\node (center\c6) at ([shift=({-\sideTwo,0})]center\c5) {};
}
{
\node (center\c7) at ([shift=({135:\sideTwo cm})]center\c6) {};
};
};
};
\foreach \ind in {3,4,5}
{
\pgfmathparse{150-120*(\ind-3)};
\node [draw=black, shape=circle, fill=black] (\c\d\ind) at ([shift=({\pgfmathresult:\radius})]center\c\d) {};
}

\node [draw=black, shape=circle, fill=black] (\c\d1) at ([yshift=\side cm]\c\d3) {};
\node [draw=black, shape=circle, fill=black] (\c\d2) at ([yshift=\side cm]\c\d4) {};
\node [draw=black, shape=circle, fill=black] (\c\d6) at ([xshift=-\side cm]\c\d5) {};
\node [draw=black, shape=circle, fill=black] (\c\d7) at ([shift=({135:\side cm})]\c\d6) {};
\foreach \ind in {2,4,5}
{
\node [scale=.4] at ([xshift=.18 cm]\c\d\ind) {$\c\d\ind$};
}
\foreach \ind in {1,3,6,7}
{
\node [scale=.4] at ([xshift=-.18 cm]\c\d\ind) {$\c\d\ind$};
}
\foreach \u/\v in {1/3,2/4,3/4,3/5,4/5,5/6,6/7}
{
\draw (\c\d\u) -- (\c\d\v);
}
}
\foreach \u/\v in {1/3,2/4,3/4,3/5,4/5,5/6,6/7}
{
\ifthenelse{\u=1}
{
\draw (\c\u\v) -- (\c\v\u);
}
{
\ifthenelse{\u=3 \AND \v=5 \OR \u=6}
{
\draw (\c\u\v) to[bend right] (\c\v\u);
}
{
\draw (\c\u\v) to[bend left] (\c\v\u);
};
};
}
}
\foreach \c in {1,2,6,7}
{
\ifthenelse{\c=1}
{
\node (center1) at ([shift=({0,\sideThree})]center3) {};
}
{
\ifthenelse{\c=2}
{
\node (center2) at ([shift=({0,\sideThree})]center4) {};
}
{
\ifthenelse{\c=6}
{
\node (center6) at ([shift=({-\sideThree,0})]center5) {};
}
{
\node (center7) at ([shift=({135:\sideThree cm})]center6) {};
};
};
};
\foreach \d in {3,4,5}
{
\pgfmathparse{150-120*(\d-3)};
\node (center\c\d) at ([shift=({\pgfmathresult:\radiuscenter cm})]center\c) {};
\foreach \ind in {3,4,5}
{
\pgfmathparse{150-120*(\ind-3)};
\node [draw=black, shape=circle, fill=black] (\c\d\ind) at ([shift=({\pgfmathresult:\radius cm})]center\c\d) {};
}
\node [draw=black, shape=circle, fill=black] (\c\d1) at ([yshift=\side cm]\c\d3) {};
\node [draw=black, shape=circle, fill=black] (\c\d2) at ([yshift=\side cm]\c\d4) {};
\node [draw=black, shape=circle, fill=black] (\c\d6) at ([xshift=-\side cm]\c\d5) {};
\node [draw=black, shape=circle, fill=black] (\c\d7) at ([shift=({135:\side cm})]\c\d6) {};
\foreach \ind in {2,4,5}
{
\node [scale=.4] at ([xshift=.18 cm]\c\d\ind) {$\c\d\ind$};
}
\foreach \ind in {1,3,6,7}
{
\node [scale=.4] at ([xshift=-.18 cm]\c\d\ind) {$\c\d\ind$};
}
\foreach \u/\v in {1/3,2/4,3/4,3/5,4/5,5/6,6/7}
{
\draw (\c\d\u) -- (\c\d\v);
}
}
\foreach \d in {1,2,6,7}
{
\ifthenelse{\d=1}
{
\node (center\c1) at ([shift=({0,\sideTwo})]center\c3) {};
}
{
\ifthenelse{\d=2}
{
\node (center\c2) at ([shift=({0,\sideTwo})]center\c4) {};
}
{
\ifthenelse{\d=6}
{
\node (center\c6) at ([shift=({-\sideTwo,0})]center\c5) {};
}
{
\node (center\c7) at ([shift=({135:\sideTwo cm})]center\c6) {};
};
};
};
\foreach \ind in {3,4,5}
{
\pgfmathparse{150-120*(\ind-3)};
\node [draw=black, shape=circle, fill=black] (\c\d\ind) at ([shift=({\pgfmathresult:\radius})]center\c\d) {};
}

\node [draw=black, shape=circle, fill=black] (\c\d1) at ([yshift=\side cm]\c\d3) {};
\node [draw=black, shape=circle, fill=black] (\c\d2) at ([yshift=\side cm]\c\d4) {};
\node [draw=black, shape=circle, fill=black] (\c\d6) at ([xshift=-\side cm]\c\d5) {};
\node [draw=black, shape=circle, fill=black] (\c\d7) at ([shift=({135:\side cm})]\c\d6) {};
\foreach \ind in {2,4,5}
{
\node [scale=.4] at ([xshift=.18 cm]\c\d\ind) {$\c\d\ind$};
}
\foreach \ind in {1,3,6,7}
{
\node [scale=.4] at ([xshift=-.18 cm]\c\d\ind) {$\c\d\ind$};
}
\foreach \u/\v in {1/3,2/4,3/4,3/5,4/5,5/6,6/7}
{
\draw (\c\d\u) -- (\c\d\v);
}
}
\foreach \u/\v in {1/3,2/4,3/4,3/5,4/5,5/6,6/7}
{
\ifthenelse{\u=1}
{
\draw (\c\u\v) -- (\c\v\u);
}
{
\ifthenelse{\u=3 \AND \v=5 \OR \u=6}
{
\draw (\c\u\v) to[bend right] (\c\v\u);
}
{
\draw (\c\u\v) to[bend left] (\c\v\u);
};
};
}
}
\foreach \u/\v in {1/3,2/4,3/4,3/5,4/5,5/6,6/7}
{
\ifthenelse{\u=1}
{
\draw (\u\v\v) -- (\v\u\u);
}
{
\ifthenelse{\u=3 \AND \v=5 \OR \u=6}
{
\draw (\u\v\v) to[bend right] (\v\u\u);
}
{
\draw (\u\v\v) to[bend left] (\v\u\u);
};
};
}
\end{tikzpicture}
\caption{The generalized Sierpi\'{n}ski graph $S(G,3)$. The base  graph $G$ is shown in Figure
\ref{FigSierpinski1y2}.}
\label{FigS(G,3)}
\end{figure}

To the best of our knowledge, \cite{SierpinskiJA} is the first
published paper studying the generalized Sierpi\'{n}ski graphs. In that article, the authors obtained closed formulae for the Randi\'{c} index of polymeric networks modelled by  generalized Sierpi\'{n}ski graphs. Also, the total chromatic number of generalized Sierpi\'{n}ski graphs has recently been  studied  in \cite{GEETHA}.
In this paper we obtain closed formulae for several parameters of generalized Sierpi\'{n}ski graphs $S(G,t)$ in terms of parameters of the base graph $G$.
In particular, we focus on the chromatic,  vertex cover,  clique and  domination numbers.

\section{Some Remarks on Trees}

Given a graph $G$, the order and size of  $S(G,t)$ is obtained in the following remark.

\begin{remark}\label{SizeOrder}
Let  graph $G$ be a graph of order $n$ and size $m$, and let $t$ be a positive integer. Then the order of $S(G,t)$ is
$n^t$ and the size is
  $m\frac{n^t-1}{n-1}.$
\end{remark}
\begin{proof}
By definition of $S(G,t)$, for any $t\ge 2$  we have that the order of $S(G,t)$ is ${\cal O}(S(G,t))=n{\cal O}(S(G,t-1))$ and ${\cal O}(S(G,1)=n$. Hence, ${\cal O}(S(G,t))=n^t$. Analogously,
the size of $S(G,t)$ is ${\cal S}(S(G,t))=n{\cal S}(S(G,t-1))+m$ and ${\cal S}(S(G,1))=m$. Then ${\cal S}(S(G,t))=m\left(n^{t-1}+n^{t-2}+\cdots +1 \right)=m\frac{n^t-1}{n-1}.$
\end{proof}

\begin{corollary}\label{S(T,t)isaTree}
For any tree $T$ and any positive integer $t$, $S(T,t)$ is a tree.
\end{corollary}

\begin{proof}
Let $n$ be the order of $T$. By the connectivity of $T$ we have that $S(T,t)$ is connected. On the other hand, by Remark \ref{SizeOrder},
$S(T,t)$ has order $n^t$ and
 size  $n^{t}-1$. Therefore, the result follows.
\end{proof}

The next result gives a formula for the number of leaves in a generalized Sierpi\'{n}ski tree. A vertex with degree one in a tree $T$ is called a \textit{leaf} and a vertex adjacent to a leaf is called a \textit{support}.
The number of leaves  of a tree $T$ will be denoted by $\varepsilon(T)$ and the set of support vertices of $T$ by $\Sup(T)$. Also, if $x\in \Sup(T)$, then  $\varepsilon_T(x)$ will denote the number of leaves of $T$ which are adjacent to $x$.

\begin{theorem}\label{End-Points-of-G}
Let $T$ be a tree of order $n$ having $\varepsilon(T)$ leaves.
   For any positive integer $t$, the number of leaves  of $S(T,t)$ is $$\varepsilon(S(T,t))=\frac{\varepsilon(T)\left(n^t-2n^{t-1}+1\right)}{n-1}.$$
\end{theorem}

\begin{proof}
Let $t\ge 2.$ For any $x \in V$, we denote by $S_x(T,t-1)$  the copy of $S(T,t-1)$ corresponding to $x$  in $S(T,t)$, \textit{i.e.}, $S_x(T,t-1)$  is the subgraph of $S(T,t)$ induced by the set $\{xw:\; w\in V^{t-1}\}$, which is isomorphic to $S(T,t-1)$.  To obtain the result, we only need to determine the contribution of  $S_x(T,t-1)$ to the number of leaves of $S(T,t)$, for all $x\in V$.  By definition of $S(T,t)$, there exists an edge of $S(T,t)$ connecting the vertex $xy...y$ of  $S_x(T,t-1)$ with the vertex $yx...x$ of $S_y(T,t-1)$ if and only if $x$ and $y$ are adjacent in $T$.  Hence, a leaf $xy...y$ of $S_x(S(T,t-1)$ is adjacent in $S(T,t)$ to a vertex $yx...x$ of  $S_y(T,t-1)$ if and  only if $y$ is a leaf of $T$ and $x$ is its support vertex. Thus, if $x\in \Sup(T)$, then the contribution of $S_x(T,t-1)$ to the number of leaves of $S(T,t)$
is $\varepsilon(S(T,t-1))-\varepsilon_T(x)$ and, if $x\not\in \Sup(T)$, then the contribution of $S_x(T,t-1)$ to the number of leaves of $S(T,t)$
is $\varepsilon(S(T,t-1)$.  Then we obtain,
\begin{align*}\varepsilon(S(T,t))&=(n-|\Sup(T)|)\varepsilon(S(T,t-1))+\sum_{x\in \Sup(T)}(\varepsilon(S(T,t-1))-\varepsilon_T(x))\\
&=n\varepsilon(S(T,t-1))-\varepsilon(T).
\end{align*}
Now, since $\varepsilon(S(T,1))=\varepsilon(T)$, we have that
$$\varepsilon(S(T,t))=\varepsilon(T)\left(n^{t-1}-n^{t-2}-\cdots -n-1 \right)=\varepsilon(T)\left( n^{t-1}-\frac{\left(n^{t-1}-1\right)}{n-1}\right).$$
Therefore, the result follows.
\end{proof}

\section{Chromatic Number and Clique Number}

The \textit{chromatic number} of a graph $G=(V,E)$, denoted by $\chi(G)$,  is the smallest number of colors needed to color the vertices of $G$ so that no two adjacent vertices share the same color.
A \textit{proper vertex-colouring} of $G$ is a map $f: V\longrightarrow \{1,2,...,k\}$   such that for any edge $\{u,v\}$  of $G$, $f(u)\ne f(v)$. The elements of $\{1,2,...,k\}$ are called
colours, the vertices of one colour form a \textit{colour class} and we say that $f$ is a $k$-\textit{colouring}.
So the chromatic number of $G$ is the minimum $k$ such that there exists a  $k$-colouring. For instance,  for any bipartite graph $G$, $\chi(G)=2.$
Since every  tree is a bipartite graph, by Corollary \ref{S(T,t)isaTree} we conclude that for any tree $T$ and any positive integer $t$, $\chi(S(T,t))=2$.

The problem of finding  chromatic number of a graph is NP-hard, \cite{Garey1979}. This suggests finding the chromatic number
for special classes of graphs or obtaining good bounds on this invariant.  As shown in \cite{Parisse}, $\chi(K_n,t)=n$. We shall show that the chromatic number of a generalized Sierpi\'{n}ski graph is determined by the chromatic number of its base graph.

\begin{theorem}\label{TheoremChromatic}
For any graph $G$ and any positive integer $t$, $$\chi(S(G,t))=\chi(G).$$
\end{theorem}

\begin{proof}
Let $w$ be a word of length $t-1$ on the alphabet $V$. By definition of $S(G,t)$, the subgraph $\langle V_w \rangle$ of $S(G,t)$ induced by the set $V_w=\{wx:\; x\in V(G)\}$ is isomorphic to $G$. Hence, $\chi(S(G,t))\ge \chi(\langle V_w \rangle)=\chi(G)$.

Now, let $f: V\longrightarrow \{1,2,...,k\}$ be a proper vertex-colouring of $G$ and let $f_1: V^t\longrightarrow \{1,2,...,k\}$ be a map defined by $f_1(wx)=f(x)$, for all $wx\in V^t$. If two vertices $wx,w'y\in V^t$ are adjacent in $S(G,t)$, then  $x$ and $y$ are adjacent in $G$. Hence, if $wx,w'y\in V^t$ are adjacent in $S(G,t)$, then  $f_1(wx)=f(x)\ne f(y)=f_1(w'y)$ and, as a consequence, $f_1$ is a  proper vertex-colouring of $S(G,t)$. Therefore, $\chi(S(G,t))\le \chi(G)$.
\end{proof}

As a direct consequence of Theorem \ref{TheoremChromatic} we deduce the following result.

\begin{corollary}
For any bipartite graph $G$ and any positive integer $t$, $S(G,t)$ is bipartite.
\end{corollary}

 A \textit{clique} of a graph $G=(V,E)$ is a subset $C\subseteq V$ such that for any pair of different vertices $v,w \in \ C$, there exists an edge $\{v,w\}\in E$, \textit{i.e.}, the subgraph induced by $C$ is complete.  The \textit{clique number} of a graph $G$, denoted by   $\omega(G)$,  is the number of vertices in a maximum clique of $G$. The chromatic number  of a graph  is equal to or greater than its clique number, \textit{i.e}.,
$\chi(G)\ge \omega(G)$.

 It is well-known that the problem of finding a maximum clique is NP-complete, \cite{Garey1979}.
 We shall show that the clique number of a generalized Sierpi\'{n}ski graph is equal to the clique number of its base graph.

\begin{theorem}\label{TheoremCliqueNumber}
For any graph $G$ of order $n$ and any positive integer $t$, $$\omega(S(G,t))=\omega(G).$$
\end{theorem}

\begin{proof}
We shall show that for any $t\ge 2$, $\omega(S (G,t))=\omega(S(G,t-1))$. Let $x\in V$ and let $\langle V_x \rangle$ be the subgraph of $S(G,t)$ induced by the set $V_x=\{xw:\; w\in V^{t-1}\}$. Since  $\langle V_x \rangle \cong S(G,t-1)$,
$\omega(S(G,t))\ge \omega(\langle V_x \rangle)=\omega(S(G,t-1)).$

Now, let $C$ be a maximum clique  of $S(G,t)$ and
 $xw_1,yw_2\in C$. If $x\ne y$, then $yw_2=yxx\ldots x$  is the only vertex not belonging to $V_x$ which is adjacent to  $xw_1=xyy\ldots y$ and, analogously, $xw_1=xyy\ldots y$  is the only vertex not belonging to $V_y$ which is adjacent to  $yw_2=yxx\ldots x$. Hence, $x\ne y$ leads to $|C|=2=\omega(S(G,t-1))$ and so $|C|>2$ leads to $x=y$  which implies that $C\subset V_x$.  Thus,  $\omega(S(G,t))=|C|\le \omega(\langle V_x \rangle)=\omega(S (G,t - 1))$.

Therefore, $\omega(S (G,t))=\omega(S(G,t-1))=\cdots=\omega(S(G,1))=\omega(G)$.
\end{proof}

\section{Vertex Cover Number and Independence Number}

 A \textit{vertex cover}  of a graph $G$ is a set of vertices such that each edge of $G$ is incident to at least one vertex of the set. The \emph{vertex cover number} of $G$, denoted by $\beta(G)$, is the smallest cardinality of a vertex cover of $G$.
 For example, in the graph $G$ of Figure \ref{FigSierpinski1y2}, $\{3,4,6\}$ is a vertex cover of minimum cardinality and so $\beta(G)=3$.

It is well-known that the problem of finding a minimum vertex cover is a classical optimization problem in computer science and is a typical example of an NP-hard optimization problem, \cite{Garey1979}.   As the next result shows, the vertex cover number of a generalized Sierpi\'{n}ski graph can be compute from the vertex cover number and the order of  the base graph.

\begin{theorem}\label{TheoremVertexCover}
For any graph $G$ of order $n$ and any positive integer $t$, $$\beta(S(G,t))=n^{t-1}\beta(G).$$
\end{theorem}

\begin{proof}
Let $w\in V^{t-1}$ be a word of length $t-1$ on the alphabet $V$. By definition of $S(G,t)$, the subgraph $\langle V_w \rangle$ of $S(G,t)$ induced by the set $V_w=\{wx:\; x\in V\}$ is isomorphic to $G$. Hence,
$$\beta(S(G,t))\ge \sum_{w\in V^{t-1}(G)}\beta(\langle V_w \rangle)=n^{t-1}\beta(G).$$

Now, let $C\subset V$ be a vertex cover  of $G$ of cardinality $|C|=\beta(G)$ and let $$C'=\{wv:\; v\in C\; {\rm and}\; w\in V^{t-1}\}.$$
Since $\langle V_w \rangle \cong G$, for any $w\in V^{t-1}$ we have that $C'_w=\{wv:\; v\in C\}\subset C'$ is a vertex cover of $\langle V_w \rangle$. In addition, if two vertices $w_1y, w_2x\in V^t$, $w_1\ne w_2$, are  adjacent in $S(G,t)$, then $x$ and $y$ are adjacent in $G$ and so $x\in C$ or $y\in C$.  Hence,  $w_1y\in C'$ or $w_2x\in C'$.
Therefore, $C'$ is a vertex cover of $S(G,t)$ and, as a consequence, $\beta(S(G,t))\le |C'|=n^{t-1}|C|= n^{t-1}\beta(G).$
\end{proof}

Recall that the largest cardinality of a set of vertices of $G$, no two of which are adjacent, is called the \emph{independence number} of $G$ and is denoted by $\alpha(G)$. For example, in the graph $G$ of Figure \ref{FigSierpinski1y2}, $\{1,2,5,7\}$ is an independen set  of maximum cardinality and so $\alpha(G)=4$.

The following well-known result, due to Gallai, states the relationship between the independence number and the vertex cover number of a graph. Such a result will provide us with another very useful result on generalized Sierpi\'{n}ski graphs.

\begin{theorem}{\rm \cite{gallai}}\label{th_gallai}
For any graph  $G$ of order $n$,
$\beta(G)+\alpha(G) = n.$
\end{theorem}

By using this result and  Theorem  \ref{TheoremVertexCover} we obtain a formula for the independence number of $S(G,t)$.

\begin{theorem}\label{TheoremIndependenceNumber}
For any graph $G$ of order $n$ and any positive integer $t$, $$\alpha(S(G,t))=n^{t-1}\alpha(G).$$
\end{theorem}

\section{Domination Number}

For a  vertex
$v$ of $G=(V,E)$, $N_G(v)$ denotes the set of neighbours that $v$ has in $G$.
A set $D\subseteq V$  is {\it dominating} in $G$ if every vertex of $V-D$ has at
least one neighbour in $D$, \textit{i.e.}, $D\cap N_G(u)\ne \emptyset$, for all $u\in V-D$. The {\it domination number} of $G$, denoted by  $\gamma (G)$, is the minimum cardinality among all dominating sets in $G.$ A   dominating set of cardinality $\gamma (G)$ is called a $\gamma (G)$-set. Let ${\cal D}(G)$ be the set of all  $\gamma (G)$-sets. The dominating set problem, which concerns testing whether $\gamma(G)\le k$ for a given graph $G$ and input $k$,  is a classical NP-complete decision problem in computational complexity theory, \cite{Garey1979}. In this section we obtain an upper bound on the domination number of $S(G,t)$ and we show that the bound is tight. 

 We define the parameter $\xi(G)$ as follows:
 $$\xi(G)=\max_{D\in {\cal D}(G)} \{|D'|:\; D'\subseteq D \mbox{ and $\langle D'\rangle$ has no isolated vertices}\}.$$
Notice that $0\le \xi(G) \le \gamma(G)$. In particular,
$\xi(G)=0$ if and only if any $\gamma (G)$-set is independent, while
$\xi(G) = \gamma(G)$
if and only if there exists a $\gamma (G)$-set whose induced subgraph has no isolated vertices.

\begin{theorem}\label{UpperBoundDomination}
For any graph   $G$   of order $n$ and any integer   $t\ge 2$,
$$\gamma(S(G,t)) \le n^{t - 2}(n\gamma(G)-\xi(G)).$$
\end{theorem}

\begin{proof}
Let  $D$ be a $\gamma(G)$-set and let  $$D_{t-1}=\{wx:\; w\in V^{t - 1}  \mbox{ and } x\in D\}.$$
If $u\in V$ is adjacent to $v\in D$, then for any  $w \in V^{t - 1}$,  we have that  $wu\in V^{t}$ is adjacent to  $wv\in D_{t-1}$ and, as a consequence, $D_{t-1}$ is a dominating set in $S(G,t).$

Now, assume  that there exists $D'\subseteq D$ such that the subgraph induced by $D'$ has no isolated vertices and $|D'|=\xi(G)$. Let define the set $$D_{t-2}=\{w'uu:\; w'\in V^{t-2} \mbox{ and } u\in D'\}.$$   We shall show that $D^*=D_{t-1}-D_{t-2}$ is a dominating set in $S(G,t) $.
To this end,  taking  $w'\in V^{t-2}$ and $u\in D'$,  we only need to show that each $x\in N_{S(G,t)}(w'uu)\cup\{w'uu\}$ is dominated by some vertex belonging to  $D^*$. 
Since $w'uu$ is dominated by $w'uv\in D^*$, for some $v\in D'\cap N_G(u)$, from now on we assume that $x\in N_{S(G,t)}(w'uu)$. Now, if $x=w'uz$, for some $z\in N_G(u)$, then $x$ is dominated by $w'zu\in D^*$, so we assume that $x=w''zz$, where $w''\in V^{t-2}$ and $z\in N_G(u)$. Thus, for $z\not \in D$ or $z\in D\cap D'$ we have that $x=w''zz$ is dominated by $w''zu\in D^*$. Finally, if 
 $z\in D-D'$, then $x=w''zz\in D^*$. 

Hence, $D_{t-1}-D_{t-2}$  is a dominating set in $S(G,t) $  and, as a consequence,
 $$\gamma(S(G,t)) \le|D_{t-1}-D_{t-2}| =n^{t-1}|D|-n^{t-2}|D'|.$$
 Therefore, the result follows.
\end{proof}

From now on $\Omega(G)$ denotes the set of vertices of degree one in $G$.

\begin{lemma}\label{Lemma2Domination}
Let $G$ be a graph such that $\gamma(G)=\beta(G)$. If there exists a unique $\gamma(G)$-set $D$, then   $|\Omega(G)\cap N_G(x)|\ge 2$, for every $x\in D$.
\end{lemma}

\begin{proof}
Let $D$ be a $\beta(G)$-set. Since $\gamma(G)=\beta(G)$, $D$ is a $\gamma(G)$-set and $V-D$ is an $\alpha(G)$-set. Assume that $D$ is the only $\gamma(G)$-set and suppose that  $|\Omega(G)\cap N_G(v)|\le 1$, for some $v\in D$.
 Let $x\in N_G(v)-D$. If $x\not \in \Omega(G)$, then there exists $v'\in D$ such that  $x\in N_G(v')$. Hence, if $\Omega(G)\cap N_G(v)=\emptyset$, then $(D-\{v\})\cup \{x\}$ is a dominating set, which is a contradiction. Also, if  $\Omega(G)\cap N_G(v)=\{y\}$, then $(D-\{v\})\cup \{y\}$ is a dominating set, which is a contradiction again. Therefore, the result follows.
\end{proof}

\begin{theorem}\label{EqualityDomination}
Let $G$ be a graph of order $n$ such that there exists a unique
$\gamma(G)$-set and $\gamma(G)=\beta(G)$. Then for any integer $t\ge
2$ $$\gamma(S(G,t)) = n^{t - 2}(n\gamma(G)-\xi(G)).$$
\end{theorem}

\begin{proof}
Let $D\subset V$ be the only $\gamma(G)$-set. As we have shown
in the proof of Theorem \ref{UpperBoundDomination},  $D^* =
D_{t-1}-D_{t-2}$ is a dominating set of $S(G,t)$ and $|D^*|= n^{t - 2}(n\gamma(G)-\xi(G)).$
 Let $D^{\pi}$ be a
dominating set of $S(G,t)$ of minimum cardinality.  If $D^* - D^{\pi} =
\emptyset$, then $|D^*| \le |D^{\pi}|$ and, as a consequence, $\gamma(S(G,t)) = |D^*|$.
 Let $w\in V^{t-1}$ and assume  that $w x \in D^* - D^{\pi}$.
Since  $x\in D$, by Lemma \ref{Lemma2Domination} we have
$|\Omega(G)\cap N_G(x)|\ge 2$. Thus, if $wx$ is not the extreme of
$\langle V_w \rangle$, then there are at least two vertices of degree
one adjacent to $wx$ in $S(G,t)$, which implies that $wx\in D^{\pi}$ and it is a contradiction.
Thus, $wx$ is the extreme vertex of $\langle
V_w\rangle$ \textit{i.e.},  $wx =w'xx$, for some $w'\in V^{t-2}$. Notice that:
\begin{enumerate}[\bf -]
\item { $w'xx$ is dominated by a vertex $w''y \in D^{\pi} $. Now, if
$w''y \in D^*$, then $x$ and $y$ are adjacent in $G$, $x \in D'$ and
$w'xx \not\in D^*$,  which is a contradiction. Hence, $w''y \in D^{\pi} - D^*$.}
\item {The only vertex in $D^* - D^{\pi}$ dominated by $w''y$ is
$w'xx$, as a vertex in $S(G, t)$ can only be adjacent to one extreme
vertex.}
\end{enumerate}
Let $S\subseteq D^{\pi} - D^*$ such that each vertex of
$D^* - D^{\pi}$ is dominated by a vertex in $S$. Then we can
define a mapping $f: S\longrightarrow D^* - D^{\pi}$, where
$f(u) = v$ means that $v$ is dominated by $u$. As $f$ is an onto
mapping, $|D^* - D^{\pi}| \le |S| \le
|D^{\pi} - D^*|$. Therefore, $|D^*| \le |D^{\pi}|$, and  we  can
conclude that $\gamma(S(G,t)) = |D^{\pi}|  =|D^*| = n^{t -
2}(n\gamma(G)-\xi(G)).$
\end{proof}

\begin{lemma}\label{Lemma1Domination}
Let $G$ be a graph of order $n$ and let $t\ge 3$ be an integer. If $\gamma(S(G,t))=n^{t-1}\gamma(G)$, then there exists a unique $\gamma(G)$-set.
\end{lemma}

\begin{proof}
Assume that $\gamma(S(G,t))=n^{t-1}\gamma(G).$ Notice that, by Theorem \ref{UpperBoundDomination}, we have that $\xi(G)=0$. Suppose, for contradiction proposes, that $A$ and $B$ are two different $\gamma(G)$-sets. In such a case, there exist $a\in A-B$ and $b\in B$ such that $b$ dominates $a$. Now, if $b\in A$, then $\xi(G)>0$, which is a contradiction. So, $b\not\in A$.   Following a procedure analogous to that used in the proof of Theorem \ref{UpperBoundDomination} we see that $A_{t-1}=\{wv:\; w\in V^{t-1}\; {\rm and}\; v\in A\}$ is a dominating set of $S(G,t)$. Hence,
$$A'=\left(A_{t-1} -( \{abb\ldots bx:\; x\in A\} \cup \{baa\ldots a\} \right) )\cup \{abb\ldots bx:\; x\in B\}$$
is a dominating set of $S(G,t)$
as any vertex in $\{abb\ldots bx:\; x\in V\}$  is dominated by some vertex in $\{abb\ldots bx:\; x\in B\}\subset A'$, $baa\ldots a$ is dominated by  $abb\ldots b\in A'$ and, for any $z\in N_G(a)$,  $ baa\ldots az$  is dominated by  $baa\ldots aza\in A'$. Thus, $\gamma(S(G,t))\le |A'|=n^{t-1}\gamma(G)-1$, which is a contradiction. Therefore, the result follows.
\end{proof}

\begin{theorem}\label{EquivDomination}
Let $G$ be a graph of order $n$ such that $\gamma(G)=\beta(G)$ and let $t\ge 3$ be an  integer. The following assertions are equivalent.
\begin{enumerate}[{\rm (a)}]
\item $\gamma(S(G,t))=n^{t-1}\gamma(G).$
\item $\xi(G)=0$ and there exists a unique $\gamma(G)$-set.
\end{enumerate}
\end{theorem}

\begin{proof}
 Assume that $\gamma(S(G,t))=n^{t-1}\gamma(G).$ By Theorem \ref{UpperBoundDomination}, we have that $\xi(G)=0$ and by Lemma \ref{Lemma1Domination} we have that there exists a unique $\gamma(G)$-set.

Now, if $\xi(G)=0$ and   there exists a unique $\gamma(G)$-set, then by Theorem \ref{EqualityDomination} we conclude that  $\gamma(S(G,t))= n^{t-1}\gamma(G).$
\end{proof}

It is ready to see that $\gamma(S(K_{1,r},2))=r+1$.  Hence, from Theorem \ref{EquivDomination} we deduce the following result.
\begin{corollary}
For any positive integers $r$ and $t$,
$$\gamma(S(K_{1,r},t))=(r+1)^{t-1}.$$
\end{corollary}

\end{document}